\documentclass[11pt]{article}

\usepackage[T1]{fontenc}
\usepackage{lmodern}
\usepackage{microtype}
\usepackage{amsmath,amssymb,amsthm}
\usepackage[margin=1in]{geometry}
\usepackage[hidelinks]{hyperref}

\newtheorem{theorem}{Theorem}[section]

\newtheorem{lemma}[theorem]{Lemma}
 
\theoremstyle{remark}

\newcommand{\F}{\mathbb F}
\newcommand{\wt}{\operatorname{wt}}
\newcommand{\tr}{\operatorname{tr}}

\newcommand{\cR}{\mathcal R}
\newcommand{\cS}{\mathcal S}
\newcommand{\cC}{\mathcal C}

\title{The Quadratic Easy Coefficients Conjecture\\
via Finite-Type Shifts and Zeta Functions}
\author{Thomas W. Cusick$^a$ \footnote{email: cusick@buffalo.edu~ORCID 0000-0002-0087-8855}
\vspace{.5cm}\\
$^a$\small Department of Mathematics,
\small University at Buffalo\\
\small 244 Mathematics Bldg.,  Buffalo, NY 14260\\}
\date{September 1, 2026}

\begin{document}

\maketitle

\begin{abstract}
We prove the Quadratic Easy Coefficients Conjecture stated as Conjecture 1
in T. W. Cusick, \emph{Recursions for quadratic rotation symmetric
functions weights}, Discrete Applied Mathematics 378 (2026), 93--101.
For an arbitrary finite sum of quadratic monomial rotation symmetric
Boolean functions, we identify the recurrent part of the rules
matrix with a signed binary de Bruijn transfer matrix $B$.  We then give a
one-step presentation of the finite-type shift associated with the Boolean
function in the symbolic-dynamics construction of Chirvasitu and Cusick.
Fourier transformation in an auxiliary $\F_2$ coordinate decomposes the
adjacency matrix of this shift into an unsigned de Bruijn block and the
signed block $B$.  Consequently the dynamical zeta function is
\[
 \zeta_{X_f}(z)=\frac{1}{\det(I-z\cR(f))},
\]
where $\cR(f)$ is the rules matrix.  This equality identifies, with their
algebraic multiplicities, the characteristic values, supplied by symbolic
dynamics, with the roots of the characteristic polynomial of the rules
matrix.  The desired easy coefficients formula follows from the trace of
$B^n$.  We also prove nonsingularity and justify the unique backward
extension of the weight recurrence.
\end{abstract}

\section{Statement of the conjecture and the theorem}

Let
\begin{equation}\label{original-function}
 f_n=\sum_{i=1}^{m}(1,t_i)_n
\end{equation}
be a quadratic rotation symmetric (RS) Boolean function, where the $t_i$ are distinct  
integers $>1$. Define
\begin{equation}\label{T-and-M}
 T=\{t_i:1\leq i\leq m\},
 \qquad M=\max_{t\in T}(t-1).
\end{equation}
For $n\geq 2M+1$ this means explicitly
\begin{equation}\label{explicit}
 f_n(x_1,\ldots,x_n)
 =\sum_{r=1}^{n}\sum_{t\in T}x_rx_{r+t-1},
\end{equation}
where the subscripts are reduced modulo $n$ and the sum is taken in
$\F_2$.

There is a general theory of \emph{rules matrices} for RS functions of any degree
described in \cite{Carx, Cus18}. This theory is much more complicated than
the degree $2$ case considered here, so we do not use those references but instead
rely on \cite[Section 2 and Appendix]{Cus26}. 
The \emph{expanded rules matrix},  as described in detail in  \cite[Lemma 3]{Cus26}, is
used to compute the weight recursions for the quadratic RS functions $f_n$
in \eqref{original-function}.  Then reducing this expanded matrix to the (smaller) rules matrix
gives a nonsingular square matrix $\cR(f)$ of size $2^M+1$, with the
distinguished characteristic root $2$.  A general form of the Easy
Coefficients Conjecture (ECC) (for RS functions of arbitrary degree) is stated as
Conjecture 6.1 in reference~\cite{CC22}.  In the quadratic case, write the
remaining roots of its characteristic polynomial, counted with their full
algebraic multiplicities, as
\begin{equation}\label{eta-list}
 \eta_1,\eta_2,\ldots,\eta_{2^M}.
\end{equation}
Then we have the following statement of the quadratic ECC: 
\begin{theorem}[Quadratic Easy Coefficients Conjecture]\label{thm:main}
Extend the weight recursion for the functions in
\eqref{original-function} backwards from $n=2M+1$ to $n=1$, and denote
the resulting sequence by $w_1,w_2,\ldots$.  Thus
\[
 w_n=\wt(f_n)\qquad(n\geq 2M+1).
\]
Then
\begin{equation}\label{ECC}
 \boxed{
 w_n=2^{n-1}-\frac12
 \left(\eta_1^n+\eta_2^n+\cdots+\eta_{2^M}^n\right),
 \qquad n\geq1.}
\end{equation}
Thus the coefficient belonging to the distinguished root $2$ is
$\frac12$, and the coefficient belonging to every occurrence of every
other characteristic root is $-\frac12$.
\end{theorem}

The fact that the roots in \eqref{eta-list} are counted with
characteristic polynomial multiplicity is essential.  The characteristic
polynomial can have repeated roots even though the minimal polynomial is
squarefree.

Applications of the ECC are given in \cite{Cus24} and \cite{Cus26}.

\section{The signed de Bruijn transfer matrix}

A \emph{state} is a binary word
$a=(a_0,\ldots,a_{M-1})\in\F_2^M$, regarded as a record of $M$
consecutive binary digits.  For a state $a$ and
$\varepsilon\in\F_2$, define
\begin{equation}\label{sigma}
 \sigma_\varepsilon(a)
 =(a_1,a_2,\ldots,a_{M-1},\varepsilon).
\end{equation}
Define a $2^M\times2^M$ integer matrix $B=B_T$ whose rows and columns are
indexed by the words in $\F_2^M$ by
\begin{equation}\label{B-definition}
 B_{a,b}=
 \begin{cases}
 \displaystyle
 (-1)^{\varepsilon\sum_{t\in T}a_{M-t+1}},
 &b=\sigma_\varepsilon(a),\quad \varepsilon\in\F_2,\\[2mm]
 0,&\text{otherwise}.
 \end{cases}
\end{equation}
The exponent is evaluated in $\F_2$.

A transition deletes the first digit, shifts the remaining digits one place,
and appends a new
digit $\varepsilon$.  When $\varepsilon=1$, its sign records all quadratic
products completed by the newly appended digit.  The matrix $B$ is a
signed binary de Bruijn transfer matrix.

\begin{lemma}\label{lem:B-orthogonal}
The matrix $B$ satisfies
\begin{equation}\label{B-orthogonal}
 BB^T=2I.
\end{equation}
Consequently $B$ is nonsingular, $B/\sqrt2$ is orthogonal, $B$ is
diagonalizable over $\mathbb C$, and every eigenvalue of $B$ has absolute
value $\sqrt2$.
\end{lemma}

\begin{proof}
Every row of $B$ has two nonzero entries, both equal to $1$ or $-1$, so its
squared norm is $2$.  The support of row $a$ is determined by the suffix
$(a_1,\ldots,a_{M-1})$.  Hence two rows having different suffixes have
disjoint supports.

Suppose two distinct rows have the same suffix.  Their indices then differ
only in $a_0$.  Since $M+1\in T$ (because  $M+1=\max ~T$), the term with $t=M+1$ in the exponent in
\eqref{B-definition} is $\varepsilon a_0$.  For the transition with
$\varepsilon=0$, the two row entries have product $1$.  For the transition
with $\varepsilon=1$, changing $a_0$ changes exactly one sign, so the two
entries have product $-1$.  The inner product of the two rows is therefore
$1-1=0$.  This proves \eqref{B-orthogonal}.  The remaining assertions
are immediate.
\end{proof}

\section{Reduction of the expanded  rules matrix}

We next identify $B$ with the  rules matrix described in \cite[Section 2]{Cus26}.
 This is the concrete bridge between the rules matrix construction
and the dynamical zeta function.

In the expanded quadratic construction, the monomial $(1,t)_n$ contributes
a binary shift register of length $t-1$.  Before adjoining the distinguished
final row and column, the expanded state space is therefore
\begin{equation}\label{expanded-state-space}
 \cS=\prod_{t\in T}\F_2^{t-1}.
\end{equation}
Write a state as
\[
 u=\bigl(u^{(t)}:t\in T\bigr),
 \qquad
 u^{(t)}=(u^{(t)}_0,u^{(t)}_1,\ldots,u^{(t)}_{t-2}).
\]
For $\varepsilon\in\F_2$, define
\begin{equation}\label{expanded-shift}
 \Sigma_\varepsilon(u)^{(t)}
 =\bigl(u^{(t)}_1,\ldots,u^{(t)}_{t-2},\varepsilon\bigr).
\end{equation}

\begin{lemma}[{Shift-register form of \cite[Appendix]{Cus26}}] \label{lem:appendix}
After a fixed simultaneous permutation of rows and columns, the top-left
part $W$ of the expanded rules matrix has the entries
\begin{equation}\label{W-entry}
 W_{\Sigma_\varepsilon(u),u}
 =(-1)^{\varepsilon\sum_{t\in T}u^{(t)}_0},
 \qquad \varepsilon\in\F_2,
\end{equation}
and all other entries are zero.
\end{lemma}

\begin{proof}
We translate equations (13)--(15) in \cite[Appendix]{Cus26} into
binary-register notation.  The integer indexing a state is divided into
blocks of lengths $t_i-1$, one block for each monomial $(1,t_i)_n$.  The
powers of $2$ collected in the set denoted by $X$ in that Appendix mark the
leading digit of each block.  Thus a subset of $X$ specifies which leading
digits are deleted during a simultaneous shift of the registers.

Equation (13) is the branch in which the newly appended digit in every
register is $0$.  It shifts every block and always gives coefficient $1$.
Equations (14) and (15) are the branch in which the newly appended digit is
$1$.  The subset of $X$ records precisely those registers whose deleted
leading digit is $1$.  Equation (14) assigns coefficient $1$ when the
cardinality of that subset is even, while equation (15) assigns coefficient
$-1$ when it is odd.  Hence the sign is
\[
 (-1)^{\#\{t:u^{(t)}_0=1\}}
 =(-1)^{\sum_{t\in T}u^{(t)}_0}.
\]
This is exactly \eqref{W-entry}.  The possible difference between the
ordering of the binary blocks in the Appendix and the ordering in
\eqref{expanded-state-space} is accounted for by the simultaneous row
and column permutation.
\end{proof}

For $a=(a_0,\ldots,a_{M-1})\in\F_2^M$, define the synchronized expanded
state $\iota(a)\in\cS$ by
\begin{equation}\label{iota}
 \iota(a)^{(t)}
 =(a_{M-t+1},a_{M-t+2},\ldots,a_{M-1}),
 \qquad t\in T.
\end{equation}
Let
\begin{equation}\label{core}
 \cC=\{\iota(a):a\in\F_2^M\}.
\end{equation}
Thus every shorter register in a state in $\cC$ is the suffix of the
maximal register of the appropriate length.

\begin{lemma}[The recurrent core]\label{lem:core}
The set $\cC$ is exactly the recurrent part of the expanded 
transition graph.  Repeated deletion of zero rows and the matching columns
removes precisely the states outside $\cC$.
\end{lemma}

\begin{proof}
The set $\cC$ is invariant.  Indeed, if
\[
 a'=(a_1,\ldots,a_{M-1},\varepsilon),
\]
then \eqref{expanded-shift} and \eqref{iota} give
\begin{equation}\label{core-invariance}
 \Sigma_\varepsilon(\iota(a))=\iota(a').
\end{equation}

Starting from an arbitrary state of $\cS$, perform $M$ transitions and
append the digits
\[
 \varepsilon_0,\varepsilon_1,\ldots,\varepsilon_{M-1}.
\]
After these transitions, all digits originally present in every register
have been deleted.  The register associated with $(1,t)_n$, of length
$t-1$, contains
\[
 (\varepsilon_{M-t+1},\ldots,\varepsilon_{M-1}),
\]
so the resulting state is
$\iota(\varepsilon_0,\ldots,\varepsilon_{M-1})\in\cC$.  Hence every path
enters $\cC$ after at most $M$ steps, and no state outside $\cC$ can lie on
a directed cycle.

Conversely, each state $\iota(a)$ lies on a directed cycle: successively
append $a_0,a_1,\ldots,a_{M-1}$.  After $M$ steps every register has returned
to its initial synchronized value.

The subgraph outside $\cC$ is therefore finite and acyclic, and there is no
edge from $\cC$ to its complement.  A finite acyclic directed graph has a
vertex with no incoming edge.  Because rows index new states and columns
index old states in \eqref{W-entry}, such a vertex gives a zero row.
Deleting that row and the matching column and repeating removes the entire
acyclic part.  No state in $\cC$ is removed because it lies on a directed
cycle and always has an incoming edge from a surviving state in $\cC$.
\end{proof}

\begin{theorem}[Rules matrix factorization]\label{rulesfactor}
After the zero-row reduction, the resulting rules matrix has the form
\begin{equation}\label{rulesblock}
 \cR(f)=
 \begin{pmatrix}
  B^T&0\\
  c&2
 \end{pmatrix}
\end{equation}
for some row vector $c$.  Consequently
\begin{equation}\label{char-factor}
 \det(xI-\cR(f))=(x-2)\det(xI-B).
\end{equation}
In particular, the numbers $\eta_1,\ldots,\eta_{2^M}$ in
\eqref{eta-list} are precisely the eigenvalues of $B$, counted with
their algebraic multiplicities.
\end{theorem}

\begin{proof}
For a synchronized state $\iota(a)$, the leading digit of its register of
length $t-1$ is $a_{M-t+1}$.  Lemma~\ref{lem:appendix} and
\eqref{core-invariance} therefore show that the restriction of $W$ to
$\cC$ has entry
\[
 W_{\iota(\sigma_\varepsilon(a)),\iota(a)}
 =(-1)^{\varepsilon\sum_{t\in T}a_{M-t+1}}.
\]
This is $B^T$, rather than $B$, because $W$ uses rows for new states and
columns for old states, while \eqref{B-definition} uses rows for old
states and columns for new states.

The final column of the expanded rules matrix is zero except for its final
entry $2$.  The final row can contain additional entries, but these form a
row vector $c$ and do not disturb the block-triangular form.  Lemma
\ref{lem:core} now gives \eqref{rulesblock}, and taking determinants
gives \eqref{char-factor}.
\end{proof}

\section{The finite-type shift as an \texorpdfstring{$\F_2$}{F2} skew product}

We now connect the preceding matrix directly with the finite-type shift defined in \cite[p. 1099]{CC22}.

Recall that if $(X,T)$ is a dynamical system and $G$ is a fiber space, a
\emph{skew product} over $(X,T)$ is a dynamical system on $X\times G$ whose
transformation has the form
\[
 S(x,g)=(Tx,\Phi(x,g)),
\]
where the update $\Phi$ of the second coordinate is allowed to depend on the
base point $x$.  In the additive group-extension case used here this has the
form
\[
 S_\phi(x,g)=(Tx,g+\phi(x)),
\]
for a $G$-valued one-step cocycle $\phi$; see
\cite[Section~2.2, pp.~17-18]{EFHN15}.  In the present construction the
fiber group is $G=\F_2$.

Consider states
\begin{equation}\label{skew-states}
 (a,s)\in\F_2^M\times\F_2.
\end{equation}
For each $\varepsilon\in\F_2$, define the transition
\begin{equation}\label{skew-transition}
 (a,s)\longmapsto
 \left(
  \sigma_\varepsilon(a),
  s+\varepsilon\sum_{t\in T}a_{M-t+1}
 \right).
\end{equation}
Let $Y_T$ denote the resulting one-step shift of finite type.  For a finite
directed graph with vertex set $V$, its \emph{ordinary zero-one adjacency
matrix} is the matrix $A=(A_{u,v})_{u,v\in V}$ defined by
\[
 A_{u,v}=\begin{cases}
 1,&\text{if there is an allowed directed edge }u\longrightarrow v,\\
 0,&\text{otherwise}.
 \end{cases}
\]
Here $V=\F_2^M\times\F_2$, and the allowed edges are precisely the
transitions just defined; we denote this adjacency matrix by $A$.

\begin{lemma}[Periodic-point interpretation]\label{lem:periodic-points}
For every $n\geq1$, the number $N_n(Y_T)$ of points fixed by the $n$th
power of the shift is
\begin{equation}\label{N-weight}
 N_n(Y_T)=2^{n+1}-2\wt(f_n),
\end{equation}
where for small $n$ the right side uses the cyclic sum in
\eqref{explicit}, with repeated monomials canceled in $\F_2$.
Consequently $Y_T$ has the same periodic-point data, and hence the same
zeta function, as the shift $X_f$ constructed in \cite[p. 1099]{CC22}.
\end{lemma}

\begin{proof}
A cyclic binary word $x=(x_1,\ldots,x_n)$ determines a walk in the base
de Bruijn graph by
\[
 a^{(r)}=(x_r,x_{r+1},\ldots,x_{r+M-1}),
 \qquad 1\leq r\leq n,
\]
where indices are reduced modulo $n$.  At the $r$th step, the newly appended
digit is $x_{r+M}$, so the increment in the second coordinate in
\eqref{skew-transition} is
\[
 x_{r+M}\sum_{t\in T}x_{r+M-t+1}.
\]
The second coordinate closes after $n$ steps exactly when (recall \eqref{explicit})
\begin{align*}
 0
 &=\sum_{r=1}^{n}\sum_{t\in T}x_{r+M}x_{r+M-t+1}\\
 &=\sum_{q=1}^{n}\sum_{t\in T}x_qx_{q+t-1}
 =f_n(x)
 \qquad\text{in }\F_2.
\end{align*}
For the second equality, for each fixed $t$ we make the cyclic change of
variable $q=r+M-t+1$.

If $f_n(x)=0$, both initial choices $s=0,1$ give closed lifted walks.  If
$f_n(x)=1$, neither choice closes.  Hence
\[
 N_n(Y_T)=2\#\{x\in\F_2^n:f_n(x)=0\}
 =2(2^n-\wt(f_n)),
\]
which is \eqref{N-weight}.  This is the fixed point count in
\cite[Lemma~3.6]{CC22}.
\end{proof}

\section{Fourier block decomposition and the zeta determinant}

Let $U$ be the unsigned binary de Bruijn matrix defined on the state space
$\F_2^M$ by
\begin{equation}\label{U-definition}
 U_{a,b}=
 \begin{cases}
 1,&b=\sigma_\varepsilon(a)	\text{ for some }\varepsilon\in\F_2,\\
 0,&\text{otherwise}.
 \end{cases}
\end{equation}

\begin{theorem}[Fourier block decomposition]\label{Fourier}
After Fourier transformation in the second coordinate
$s\in\F_2$, the adjacency matrix $A$ of $Y_T$ is similar over
$\mathbb R$ to the block diagonal matrix
\begin{equation}\label{A-decomposition}
 A\sim U\oplus B
 =\begin{pmatrix}
   U&0\\
   0&B
  \end{pmatrix}.
\end{equation}
Here $\oplus$ denotes the direct sum of matrices, and $\sim$ denotes
similarity.  Consequently
\begin{equation}\label{A-determinant}
 \det(I-zA)=(1-2z)\det(I-zB).
\end{equation}
\end{theorem}

\begin{proof}
Put
\[
 \phi(a,\varepsilon)
 =\varepsilon\sum_{t\in T}a_{M-t+1}\in\F_2.
\]
The two characters of the additive group $\F_2$ are
\[
 \chi_0(s)=1,
 \qquad
 \chi_1(s)=(-1)^s.
\]
For $a\in\F_2^M$ and $r\in\F_2$, define the normalized Fourier basis
vector
\[
 f_{a,r}
 =\frac1{\sqrt2}\sum_{s\in\F_2}(-1)^{rs}e_{(a,s)}.
\]
Thus
\[
 f_{a,0}=\frac1{\sqrt2}\bigl(e_{(a,0)}+e_{(a,1)}\bigr),
 \qquad
 f_{a,1}=\frac1{\sqrt2}\bigl(e_{(a,0)}-e_{(a,1)}\bigr).
\]
Let $P$ be the orthogonal change-of-basis matrix whose columns are these
vectors, ordered first by all pairs $(a,0)$ and then by all pairs $(a,1)$.
We now compute every entry of $P^{-1}AP=P^TAP$.

For $a,b\in\F_2^M$ and $r,q\in\F_2$, the definition of the adjacency
matrix gives
\begin{align*}
 (P^{-1}AP)_{(a,r),(b,q)}
 &=\frac12\sum_{s,t\in\F_2}
   (-1)^{rs+qt}A_{(a,s),(b,t)}.
\end{align*}
If $b\ne\sigma_\varepsilon(a)$ for both choices of $\varepsilon$, this
entry is zero.  Otherwise there is a unique $\varepsilon\in\F_2$ such
that $b=\sigma_\varepsilon(a)$, and the transition
\eqref{skew-transition} requires $t=s+\phi(a,\varepsilon)$.  Hence
\begin{align*}
 (P^{-1}AP)_{(a,r),(b,q)}
 &=\frac12\sum_{s\in\F_2}
   (-1)^{rs+q(s+\phi(a,\varepsilon))}\\
 &=(-1)^{q\phi(a,\varepsilon)}
   \frac12\sum_{s\in\F_2}(-1)^{(r+q)s}.
\end{align*}
The last sum is $2$ when $r=q$ and $0$ when $r\ne q$.  Therefore the
off-diagonal Fourier blocks vanish.  When $r=q=0$, the entry is $1$
precisely when $b=\sigma_\varepsilon(a)$, so the trivial-character block
is exactly $U$ by \eqref{U-definition}.  When $r=q=1$, the entry is
\[
 (-1)^{\phi(a,\varepsilon)}
 =(-1)^{\varepsilon\sum_{t\in T}a_{M-t+1}},
\]
so the nontrivial-character block is exactly $B$ by
\eqref{B-definition}.  Thus
\[
 P^{-1}AP=
 \begin{pmatrix}
  U&0\\
  0&B
 \end{pmatrix}
 =U\oplus B,
\]
which proves \eqref{A-decomposition}.  Notice that, with the explicit
row-source conventions in \eqref{B-definition}, \eqref{U-definition},
and the definition of $A$, no transpose is needed.  If one instead uses
the opposite adjacency convention, all three matrices are transposed;
this does not affect their characteristic polynomials or the determinants
below.

There are exactly $2^n$ closed walks of length $n$ in the unsigned binary
de Bruijn graph, one for each cyclic binary word of length $n$.  Thus
\[
 \tr(U^n)=2^n\qquad(n\geq1).
\]
Using the formal identity
\begin{equation}\label{trace-log}
 -\log\det(I-zC)=\sum_{n\geq1}\tr(C^n)\frac{z^n}{n}
\end{equation}
for a square matrix $C$, we obtain
\[
 -\log\det(I-zU)
 =\sum_{n\geq1}2^n\frac{z^n}{n}
 =-\log(1-2z).
\]
Both determinants have constant term $1$, so
$\det(I-zU)=1-2z$.  Taking determinants in
\eqref{A-decomposition} now proves \eqref{A-determinant}.
\end{proof}

Recall that the Artin--Mazur zeta function of a finite-type shift $Y$ is (see \cite[Definition 2.4]{CC22})
\begin{equation}\label{zeta-definition} 
 \zeta_Y(z)
 =\exp\left(\sum_{n\geq1}N_n(Y)\frac{z^n}{n}\right).
\end{equation}
For a one-step shift with adjacency matrix $A$, the standard determinant
formula is (see \cite[Theorem 2.5]{CC22})
\begin{equation}\label{zeta-adjacency}
 \zeta_Y(z)=\frac1{\det(I-zA)}.
\end{equation}

We shall call $\det(I-z\cR(f))$ the \emph{rules determinant}. The next theorem says this determinant shows up
in the zeta function.

\begin{theorem}[Zeta denominator equals the rules determinant]
\label{thm:zeta-rules}
For the quadratic RS family $f$,
\begin{equation}\label{zeta-rules}
 \boxed{
 \zeta_{X_f}(z)
 =\zeta_{Y_T}(z)
 =\frac1{(1-2z)\det(I-zB)}
 =\frac1{\det(I-z\cR(f))}.}
\end{equation}
\end{theorem}

\begin{proof}
The equality $\zeta_{X_f}=\zeta_{Y_T}$ follows from
Lemma~\ref{lem:periodic-points}.  Equations
\eqref{zeta-adjacency} and \eqref{A-determinant} give the middle
expression in \eqref{zeta-rules}.  Finally, Theorem
\ref{rulesfactor}, with $x$ replaced by $1/z$ and the appropriate
power of $z$ cleared, gives
\[
 \det(I-z\cR(f))=(1-2z)\det(I-zB).
\]
This proves the final equality.
\end{proof}

Theorem~\ref{thm:zeta-rules} is the crucial fact.  The general
symbolic-dynamics zeta function theorem (see \cite[Theorem 2.5]{CC22}) produces algebraic integers as characteristic
values of a finite-type shift.  Equation \eqref{zeta-rules} shows that,
in the quadratic case, these characteristic values are exactly the roots of
the characteristic polynomial of the rules matrix, with their
algebraic multiplicities.  Thus no assumption that the characteristic polynomial has distinct roots
(as was made in \cite[Section 4]{Cus26}, where the ECC was proved under that assumption) is
needed.

\section{The trace identity and easy coefficients}

\begin{lemma}[Signed trace identity]\label{lem:trace-identity}
For every $n\geq1$,
\begin{equation}\label{trace-identity}
 \tr(B^n)
 =\sum_{x\in\F_2^n}(-1)^{f_n(x)}.
\end{equation}
In particular, for $n\geq2M+1$,
\begin{equation}\label{weight-trace}
 \wt(f_n)=2^{n-1}-\frac12\tr(B^n).
\end{equation}
\end{lemma}

\begin{proof}
Expanding the trace as a product of matrix entries gives a sum over the
closed walks of length $n$ in the binary de Bruijn graph.  Such closed walks
are in bijection with cyclic words $x=(x_1,\ldots,x_n)\in\F_2^n$, since
the state at time $r$ must be $(x_r,\ldots,x_{r+M-1})$.  By
\eqref{B-definition}, the sign accumulated around the closed walk
belonging to $x$ is
\begin{align*}
 \prod_{r=1}^{n}
 (-1)^{x_{r+M}\sum_{t\in T}x_{r+M-t+1}}
 &=(-1)^{\sum_{r=1}^{n}\sum_{t\in T}
                x_{r+M}x_{r+M-t+1}}\\
 &=(-1)^{\sum_{q=1}^{n}\sum_{t\in T}x_qx_{q+t-1}}
 =(-1)^{f_n(x)}.
\end{align*}
This proves \eqref{trace-identity}.  For every Boolean function $g$ in
$n$ variables,
\[
 \sum_{x\in\F_2^n}(-1)^{g(x)}
 =\#\{x:g(x)=0\}-\#\{x:g(x)=1\}
 =2^n-2\wt(g).
\]
Applying this with $g=f_n$ proves \eqref{weight-trace}.
\end{proof}

Let $\eta_1,\ldots,\eta_{2^M}$ be the eigenvalues of $B$, counted with
algebraic multiplicity.  For every square complex matrix, diagonalizable or
not,
\begin{equation}\label{trace-eigenvalues}
 \tr(B^n)=\eta_1^n+\cdots+\eta_{2^M}^n.
\end{equation}
For example, this follows immediately by putting $B$ into upper-triangular
Schur form.  Combining \eqref{weight-trace} and
\eqref{trace-eigenvalues} gives
\begin{equation}\label{ECC-large-n}
 \wt(f_n)=2^{n-1}-\frac12
 \left(\eta_1^n+\cdots+\eta_{2^M}^n\right),
 \qquad n\geq2M+1.
\end{equation}
By Theorem~\ref{rulesfactor}, these $\eta_j$ are the
nontrivial characteristic roots of the rules matrix, with the
multiplicities required in Conjecture~1.

\section{Unique backward extension}

It remains to verify that \eqref{ECC-large-n} gives the sequence obtained
by extending the weight recursion backwards.

Lemma~\ref{lem:B-orthogonal} shows that $B$ is nonsingular and that all its
eigenvalues have absolute value $\sqrt2$.  Hence $2$ is not an eigenvalue of
$B$.  The rules matrix in \eqref{rulesblock} is therefore nonsingular.
Moreover, it is similar to the block diagonal matrix
\begin{equation}\label{block-diagonal}
 \begin{pmatrix}B^T&0\\0&2\end{pmatrix}.
\end{equation}
Indeed, since $B^T-2I$ is invertible, a lower triangular shear eliminates
the row $c$.  This also shows directly that $\cR(f)$ is diagonalizable,
although only nonsingularity is needed below.

Define for every $n\geq1$
\begin{equation}\label{z-sequence}
 z_n=2^{n-1}-\frac12
 \left(\eta_1^n+\cdots+\eta_{2^M}^n\right).
\end{equation}
Because $z_n$ is a linear combination of powers of the roots of the rules
matrix, it satisfies the recurrence determined by the characteristic polynomial
of $\det(I-z\cR(f))$.  Equivalently, this follows from Cayley--Hamilton after
taking traces of powers of $B$.

By \eqref{ECC-large-n},
\begin{equation}\label{tail-agreement}
 z_n=\wt(f_n)\qquad(n\geq2M+1).
\end{equation}
The constant coefficient of the recurrence polynomial is nonzero because
$\cR(f)$ is nonsingular.  Therefore the recurrence can be solved uniquely
backwards: the earliest term in any recurrence relation is determined by
the later terms.  There is consequently only one recurrence sequence that
agrees with the actual weights for all sufficiently large $n$.  By
\eqref{tail-agreement}, that sequence is $z_n$.  Formula
\eqref{z-sequence} is exactly \eqref{ECC}, completing the proof of
Theorem~\ref{thm:main}.

\section{Computing the rules matrix} \label{compute}
We begin with a set 
\begin{equation} \label{Qm}
 Q_m = \{(1,t_1), (1,t_2), \ldots, (1,t_m)\}
 \end{equation}
of quadratic RS functions and we shall describe the program for computing the rules matrix for the sum of those $m$ functions. 
It is convenient to refer to the final rules matrix with the last row and last column omitted as the \emph{Rules matrix}. The next theorem (it is \cite[Theorem 2]{Cus26}]) gives some important properties of the Rules matrix. Note that the Rules matrix has an especially simple form if 
\begin{equation} \label{ti}
t_1>t_2> \ldots >t_m
\end{equation}
is true.
\begin{theorem} \label{thm2}
Every row and every column in the Rules matrix for the sum of the functions in \eqref{Qm} has exactly two nonzero entries, 
which are either $1, 1$ or $1, -1$.  There is an equal number of rows with two $1$'s and with a $-1.$  If a row with $1,1$ is given, there is another row  with $1,-1$ or $-1, 1$ in the same columns in which the two $1$’s appear. Also if
a row with entries $1,-1$ is given, there is another row with $1,1$ in the same columns 
in which the $1$ and $-1$ appear.  We call the two rows with nonzero entries in the
same columns \emph{paired~rows}. If the functions in $Q_m$ satisfy \eqref{ti}, then any two paired rows
in the Rules matrix are consecutive, with the first nonzero element in row $2i-1, ~1 \leq i \leq 2^{k-1},$
where $k = \text{max} ~(t_i-1),$ occurring in position $i.$  Every row has one of its two nonzero entries in its first half and the other in its second half, with the entries in each half the same distance away from the first element in the respective halves.
\end{theorem}

If \eqref{ti} holds, Theorem \ref{thm2} implies that the Rules matrix always has two diagonal stripes made up of paired nonzero entries,
one stripe in each of the left and right halves of the matrix. Looking at the following example will be helpful in
understanding the content of Theorem \ref{thm2}. Note that in the example all paired rows are consecutive, as guaranteed by Theorem \ref{thm2}.

Let $Q_1=\{(1,4)\}.$  Then the rules matrix (size $9$) is
\begin{equation*}
   \left(
   \begin{matrix}
    1& 0  & 0  & 0  &  1 & 0  & 0  & 0  & 0 \\
    1& 0  & 0  & 0  & -1 & 0  & 0  & 0  & 0 \\
    0& 1  & 0  & 0  &  0 & 1  & 0  & 0  & 0 \\
    0& 1  & 0  & 0  &  0 &-1  & 0  & 0  & 0 \\
    0& 0  & 1  & 0  &  0 & 0  & 1  & 0  & 0 \\
    0& 0  & 1  & 0  &  0 & 0  &-1  & 0  & 0 \\
    0& 0  & 0  & 1  &  0 & 0  & 0  & 1  & 0 \\
    0& 0  & 0  & 1  &  0 & 0  & 0  &-1  & 0 \\
    0& 0  & 0  & 0  &  1 & 1  & 1  & 1  & 2  
 \end{matrix}
 \right)
\end{equation*}
The minimal polynomial for this matrix is $$x^7-2x^6-8x+16 = (x-2)(x^2-2)(x^4+2x^2+4)$$
and the characteristic polynomial is $$x^9-2x^8-2x^7+4x^6-8x^3+16x^2+16x-32 = (x-2)(x^2-2)^2(x^4+2x^2+4).$$

To obtain the rules matrix, we first compute the expanded rules matrix. Define a sequence $y_k,~0 \leq k \leq m$ by 
$$y_0=0, y_k = \sum_{i=1}^k (t_{m-i+1}-1).$$
Define
$$X= \{2^{y_k-1}:~1 \leq k \leq m\}=\{x_i:~1 \leq i \leq m\}.$$
Define a function $F: \{subsets ~of~ X\} \rightarrow \mathbb{R},$ which maps a subset $u$ of $X$  to the sum of its elements, by
$$F(u) = \sum_{ x_i \in u}.$$
Given a positive integer $b= \sum_{i=0}^B~ g_b(i) 2^i$ in base $2,$ define its $k$-th bit, using $k=0, 1, ..., B,$ counting from the right, by
$$g_b(k) = \left[ {\frac{b}{2^k}} \right] - 2 \left[{\frac{b}{2^{k+1}}}\right].$$
Define 
$$S=\{F(u):~u \in P(X)\} = F[P(X)],$$
where $P(X)$ is the power set of all subsets of $X.$

Let the entries in the expanded matrix of size $2^{y_m}+1$ be denoted by $e_{i,j}$ where
$0 \leq i,~j \leq 2^{y_m}.$ Then the matrix entries are defined as follows for $i,j <2^{y_m}:$
\begin{equation} \label{E1}
e_{i,j} = 1 ~ \text{if}~ g_i(y_0)= \ldots =g_i(y_{m-1})=0~\text{and}~j=\left[ {\frac{i}{2}} \right]+s~
\text{for~some}~s \in S
\end{equation}

\begin{equation} \label{E2}
e_{i,j} = 1 ~ \text{if~all}~ g_i(y_k)=1~\text{and}~j=\left[ {\frac{i}{2}} \right]-f(x)~
\text{for~some}~f(x) \in S,~|x|~\text{even} 
\end{equation}

\begin{equation} \label{E3}
e_{i,j} =-1~\text{if~all}~ g_i(y_k)=1~\text{and}~j=\left[ {\frac{i}{2}} \right]-f(x)~
\text{for~some}~f(x) \in S~|x|~\text{odd} 
\end{equation}

For the last row, $e_{2^{y_m},j}=1$ if \eqref{E3} holds and otherwise  $e_{2^{y_m},j}=0.$ For the last column, all entries are $0$ except  $e_{2^{y_m},2^{y_m}}=2.$  This is the result of the fact that the minimal and characteristic polynomials of the expanded matrix have one integer root $2$ 

The expanded rules matrix is then reduced in size by repeatedly deleting row $i$ and column $i,$ where row $i$ is a zero row, so that the output matrix and expanded rules matrix have the same characteristic and minimal polynomials. This process is continued until all zero rows are removed, and that final matrix is the rules matrix.  The Rules matrix is the rules matrix with the last row and last column removed.

It is important to keep in mind (as stated in Theorem \ref{thm2}) that the program places exactly two numbers (each is either $1$ or $-1$) in every row and column of the Rules matrix, and all other entries in the Rules matrix are zero, so the rules matrix is very sparse.

\end{document}